\documentclass[12pt,a4paper,english]{article}

\usepackage[latin1]{inputenc}
\usepackage[english]{babel}
\usepackage{graphicx}
\usepackage{amssymb, amsmath, amsthm}

\usepackage{fullpage}

\newtheorem{theorem}{Theorem}

\newtheorem{remark}{Remark}
\newtheorem{lemma}{Lemma}

\newtheorem{example}{Example}

\begin{document}
\title{\Large\textbf{Existence and uniqueness of solutions in the Lipschitz space of a functional equation and its application to the behavior of the paradise fish\footnote{This is an accepted version of the manuscript published in \textit{Applied Mathematics and Computation} \textbf{477} (2024), 128798 with DOI: \texttt{https://doi.org/10.1016/j.amc.2024.128798}}}}
\author{Josefa Caballero$^{\dagger}$, {\L}ukasz P{\l}ociniczak$^\sharp$, Kishin Sadarangani$^{\dagger}$}
\date{}
\maketitle
\small{\begin{center}{
			$^{\dagger}$Departamento de Matem\'aticas, Universidad de Las Palmas de Gran Canaria, \\Campus de Tafira Baja, $35017$ Las Palmas de Gran Canaria, Spain.}\end{center}}
\small{\begin{center}{
			$^\sharp$Faculty of Pure and Applied Mathematics, Wroclaw University of Science and Technology, Wyb. Wyspia\'nskiego 27, 50-370 Wroclaw, Poland.
			\vspace{10pt}\\\texttt{e-mail: josefa.caballero@ulpgc.es, lukasz.plociniczak@pwr.edu.pl,   kishin.sadarangani@ulpgc.es}}\end{center}}

\begin{abstract}
	In this paper, we examine the solvability of a functional equation in a Lipschitz space. As an application, we use our result to determine the existence and uniqueness of solutions to an equation describing a specific type of choice behavior model for the learning process of the paradise fish. Finally, we present some concrete examples where, using numerical techniques, we obtain approximations to the solution of the functional equation. As the straightforward Picard's iteration can be very expensive, we show that an analytical suboptimal least-squares approximation can be chosen in practice, resulting in very good accuracy. 
\end{abstract}




\section{Introduction}
The main purpose of this paper is to study the existence and uniqueness of a certain functional equation in the Lipschitz space. The motivation of this work appears in \cite{2n}, where the authors studied the solvability of the equation
\begin{equation}\label{ec1n}
	f(x)=xf(\alpha x +1-\alpha)+(1-x)f(\beta x),
\end{equation}
for any $x\in [0,1]$, where $f:[0,1]\to \mathbb{R}$ is an unknown function satisfying $f(0)=0$ and $f(1)=1$ with $0<\alpha\leq \beta < 1$. In \cite{2n}, it is explained in detail that (\ref{ec1n}) is a suitable mathematical model for the learning process of paradise fish. In that case, the solution sought $f$ represents the probability distribution function. The functional equation \eqref{ec1n} was devised to model an experiment conducted in \cite{bush1956two}. In this work, authors investigated learning capabilities of paradise fish and their behavior in a controlled environment. In particular, the fish were subjected to trials in which they had to choose between two exits in a channel. If a fish chose one side $75\%$ of the trials, it was rewarded. The authors investigated how this reinforcement affects the learning process of the animals. Similar equations are also important in different fields in learning theory \cite{bush1956two}.

A similar equation to \eqref{ec1n} was first considered in \cite{lyubich1973functional} where Schauder's fixed point theorem was used to prove the existence of solutions. The additional assumption was the requirement that the solution can be expanded in a particular power series. This issue was further investigated in \cite{istruactescu1976functional} where the author proved the existence and uniqueness with the Banach contraction principle and Picard's iteration. Further generalization along the existence and uniqueness result was provided in \cite{berinde2015functional}.

In the present paper, we investigate the existence and uniqueness of solutions to the following functional equation
\begin{equation}\label{ec2n}
	f(x)=\varphi(x)f(\varphi_1(x))+(1-\varphi(x))f(\varphi_2(x)), \quad x\in[0,1],
\end{equation}
in the Lipschitz space, under certain conditions on $\varphi$ and $\varphi_{1,2}$. It is clear that (\ref{ec2n}) is a generalization of (\ref{ec1n}). We find sufficient conditions for the unique solution to exist and show that they reduce to the previously known result for the paradise fish equation. Moreover, by using numerical and analytical techniques, we obtain some approximations to the unique solution. We also show that using Picard's iteration suggested in \cite{2n} leads to an exponential increase in the computation time needed for each subsequent approximation, which renders the numerical scheme unfavorable in realistic situations. We are able to devise a simple analytical approximation that gives the practitioner a very accurate function to work with. 

\section{Preliminaries}
We start this section by collecting some definitions and results about the Lipschitz space. This material appears in \cite{1n}. By $H_1[0,1]$, we denote the space of functions $f:[0,1]\to \mathbb{R}$ satisfying 
\begin{equation}
	\sup \left\{\frac{|f(x)-f(y)|}{|x-y|}: \ x,y\in [0,1],\ x\neq y\right\}<\infty.
\end{equation}
It can be proved that $H_1 [0,1]$ can be normed by
\begin{equation}
	\|f\|=|f(0)|+\sup \left\{\frac{|f(x)-f(y)|}{|x-y|}: \ x,y\in [0,1],\ x\neq y\right\},
\end{equation}
after which $(H_1[0,1],\|\cdot\|)$ becomes a Banach space. It is easily seen that $H_1[0,1]\subset C[0,1]$ and, for any $f\in H_1[0,1]$, $\|f\|_{\infty}\leq \|f\|$, where $\|\cdot\|_{\infty}$ denotes the usual supremum norm
\begin{equation}
	\|f\|_{\infty}=\sup\{|f(x)|:\ x\in [0,1]\}.
\end{equation}
In our study, we consider a particular subset of $H_1[0,1]$ given by
\begin{equation}
	D^{0,1}[0,1]=\{f\in H_1[0,1]:\; f(0)=0, \; f(1)=1\}.
\end{equation}
Since the convergence in $\|\cdot\|$ in $H_1[0,1]$ is dominated by the convergence in $\|\cdot\|_{\infty}$, it is clear that $D^{0,1}[0,1]$ is a closed subset of $H_1[0,1]$. Therefore, $(D^{0,1}[0,1],d)$ is a complete metric space,  with $d$ defined by
\begin{equation}
	d(f,g)=\|f-g\|\quad \text{for any} \quad f,g\in D^{0,1}[0,1].
\end{equation}
We can list several examples of functions from $ D^{0,1}[0,1]$.
\begin{example}
	Let $\varphi:[0,1]\to \mathbb{R}$ be the function given by $\varphi(x)=x^n$ with $n\in \mathbb{N}$. It is clear that $\varphi\in D^{0,1}[0,1]$ and $\|\varphi\|\leq n$. In fact, for any $x,y\in [0,1]$, $x\neq y$, we have
	\begin{equation}
		\frac{|\varphi(x)-\varphi(y)|}{|x-y|}=\frac{|x^n-y^n|}{|x-y|}=|x^{n-1}+x^{n-2}y+\cdots+y^{n-1}|\leq n.
	\end{equation}
\end{example}
\begin{example}
	Let $\varphi:[0,1]\to \mathbb{R}$ be the function given by $\varphi (x)=\sin \left(\frac{\pi}{2}x\right)$. It is easily seen that $\varphi\in D^{0,1}[0,1]$ and $\|\varphi\|\leq \frac{\pi}{2}$.
\end{example}
\begin{example}
	Let $\varphi$ be such that $\varphi\in D^{0,1}[0,1]$ and $\rho \in D^1[0,1]=\{\varphi\in H_1[0,1]:\ \varphi(1)=1\}$ then, one can compute $\varphi\cdot \rho \in D^{0,1}[0,1]$.
\end{example}
\begin{example}
	Suppose that $\varphi,\rho \in D^{0,1}[0,1]$ then $\frac{1}{2}(\varphi +\rho)\in D^{0,1}[0,1]$.
\end{example}

\section{Existence and uniqueness}
The main result of the paper is to prove the existence and uniqueness of solutions of the following functional equation 
\begin{equation}\label{ec1}
	f(x)=\varphi(x)f(\varphi_1(x))+(1-\varphi(x))f(\varphi_2(x)),
\end{equation}
where 
\begin{equation}
	\begin{cases}
		\varphi \in D^{0,1}[0,1], & \\
		H_1[0,1] \ni \varphi_{1,2}:[0,1]\to [0,1], & \varphi_1(1)=1, \; \varphi_2(0)=0,
	\end{cases}
\end{equation}
and $f:[0,1]\to \mathbb{R}$ is the sought function. The solution to (\ref{ec1}) is searched for in the complete metric space $(D^{0,1}[0,1], d)$. To prove our main result, we need the following lemma.
\begin{lemma}\label{lema2}
	For any $\varphi \in D^{0,1}[0,1]$ and $x\in [0,1]$, we have
	\begin{equation}
		|\varphi(x)-1|\leq \|\varphi\|.
	\end{equation}
\end{lemma}
\begin{proof}
	Since  $\varphi\in D^{0,1}[0,1]$, $\varphi(1)=1$, we can compute
	\begin{equation}
		\begin{split}
			|\varphi(x)-1|&=|\varphi(x)-\varphi(1)|=\frac{|\varphi(x)-\varphi(1)|}{|x-1|}\cdot |x-1|\\
			&\leq \sup \left\{ \frac{|\varphi(p)-\varphi(q)|}{|p-q|}:\ p , q \in [0,1], \ p\neq q\right\}\cdot |x-1|\\
			&\leq  \sup \left\{ \frac{|\varphi(p)-\varphi(q)|}{|p-q|}:\ p , q \in [0,1], \ p\neq q \right\}\\
			&\leq \|\varphi\|-|\varphi(0)|=\|\varphi\|,
		\end{split}
	\end{equation} 
	where we have used that $|1-x|\leq 1$ and $\varphi(0)=0$. This completes the proof.
\end{proof}

We are ready to present our result for the existence of a solution to (\ref{ec1}).
\begin{theorem}\label{teo1}
	Let $T$ be the operator defined on $D^{0,1}[0,1]$ by
	\begin{equation}
		Tf(x)=\varphi(x)f(\varphi_1(x))+(1-\varphi(x))f(\varphi_2(x)),
	\end{equation}
	for $f, \varphi\in D^{0,1}[0,1]$, $\varphi_1, \varphi_2:[0,1]\to [0,1]$, with $\varphi_1,\varphi_2 \in H_1[0,1]$ and $\varphi_1 (1)=1$ and $\varphi_2 (0)=0$. Then, we have
	\begin{itemize}
		\item[a)] $Tf\in D^{0,1}[0,1]$,
		\item[b)]$\|Tf\|\leq \left[2\|\varphi\|(\|\varphi_1\|+\|\varphi_2\|)-\|\varphi\|\varphi_1 (0)\right]\|f\|$,
		\item[c)] for any $f,g\in D^{0,1}[0,1]$, 
		\begin{equation}
			d(Tf,Tg)\leq \left[2\|\varphi\|(\|\varphi_1\|-\varphi_1 (0)+\|\varphi_2\|)\right]d(f,g).
		\end{equation}
	\end{itemize}
\end{theorem}
\begin{proof}
	\begin{itemize}
		\item[a)] First, since $\varphi(0)=f(0)=0$, $\varphi(1)=f(1)=1$, $\varphi_1(1)=1$ and $\varphi_2 (0)=0$, it follows that
		\begin{equation}
			Tf(0)=\varphi(0)f(\varphi_1(0))+(1-\varphi(0))f(\varphi_2(0))=f(0)=0,
		\end{equation}
		and
		\begin{equation}
			Tf(1)=\varphi(1)f(\varphi_1(1))+(1-\varphi(1))f(\varphi_2(1))=f(\varphi_1(1))=f(1)=1.
		\end{equation}
		To prove that $Tf\in H_1[0,1]$, we will estimate the following quotient 
		\begin{equation}
			\frac{|(Tf)(x)-(Tf)(y)|}{|x-y|},
		\end{equation} 
		for $x,y\in [0,1]$, $x\neq y$. In fact, by writing out the equation and using the triangle inequality we obtain
		\begin{equation}
			\begin{split}
				&\frac{|(Tf)(x)-(Tf)(y)|}{|x-y|}\\
				& =\frac{\left|\varphi(x)f(\varphi_1(x))+(1-\varphi(x))f(\varphi_2(x))-\varphi(y)f(\varphi_1(y))-(1-\varphi(y))f(\varphi_2(y))\right| }{|x-y|}\\ 
				&\leq \frac{|\varphi(x)f(\varphi_1(x))-\varphi(x)f(\varphi_1(y))|}{|x-y|}+ \frac{\left|(1-\varphi(x))f(\varphi_2(x))-(1-\varphi(x))f(\varphi_2(y))\right|}{|x-y|}\\	
				&+ \frac{\left|\varphi(x)f(\varphi_1 (y))-\varphi(y)f(\varphi_1(y))\right|}{|x-y|}+ \frac{\left|(1-\varphi(x))f(\varphi_2 (y))-(1-\varphi(y))f(\varphi_2(y))\right|}{|x-y|}.
			\end{split}
		\end{equation}
		Further, we can extract the quotients in each term to arrive at
		\begin{align}\label{ec10}
			\frac{|(Tf)(x)-(Tf)(y)|}{|x-y|} &\leq |\varphi(x)|\frac{|f(\varphi_1(x))-f(\varphi_1(y))|}{|\varphi_1(x)-\varphi_1(y)|}\cdot\frac{|\varphi_1(x)-\varphi_1(y)|}{|x-y|}\\
			&+ |1-\varphi(x)|\frac{\left|f(\varphi_2(x))-f(\varphi_2(y))\right|}{|\varphi_2(x)-\varphi_2(y)|}\cdot\frac{|\varphi_2(x)-\varphi_2(y)|}{|x-y|}\\
			&+ \frac{|\varphi(x)-\varphi(y)|}{|x-y|}|f(\varphi_1(y))|+\frac{|\varphi(x)-\varphi(y)|}{|x-y|}|f(\varphi_2(y))|.
		\end{align}
		As $f\in D^{0,1}[0,1]$ and $f(0)=0$ we consequently have
		\begin{align}\label{ec5}
			\frac{|f(\varphi_1(x))-f(\varphi_1(y))|}{|\varphi_1(x)-\varphi_1(y)|}&\leq \left\{ \frac{|f(x)-f(y)|}{|x-y|}:\ x,y\in [0,1],\ x\neq y\right\}\leq \|f\|.
		\end{align}
		Moreover, $\varphi\in H_{1}[0,1]$ and $\|\varphi\|_{\infty}\leq \|\varphi\|$ imply that
		\begin{equation}\label{ec2}
			|\varphi(x)|\leq \|\varphi\|.
		\end{equation}
		On the other hand, because $\varphi_1 \in H_1[0,1]$, we deduce that
		\begin{equation}\label{ec3}
			\frac{|\varphi_1(x)-\varphi_1(y)|}{|x-y|}\leq \|\varphi_1\|-\varphi_1(0).
		\end{equation}
		Furthermore, as $\varphi\in D^{0,1}[0,1]$ by Lemma \ref{lema2} we have
		\begin{equation}\label{ec4}
			|1-\varphi(x)|\leq \|\varphi\|.
		\end{equation}
		Using an argument similar to the one used to prove the inequality (\ref{ec5}), we can obtain
		\begin{equation}\label{ec6}
			\frac{|f(\varphi_2(x))-f(\varphi_2(y))|}{|\varphi_2(x)-\varphi_2(y)|}\leq \left\{ \frac{|f(x)-f(y)|}{|x-y|}:\ x,y\in [0,1],\ x\neq y\right\}\leq \|f\|.
		\end{equation}
		Furthermore, as $\varphi_2\in H_1[0,1]$ and $\varphi_2(0)=0$, we deduce that 
		\begin{equation}\label{ec7}
			\frac{|\varphi_2(x)-\varphi_2(y)|}{|x-y|}\leq \|\varphi_2\|,
		\end{equation}
		and, since $f\in D^{0,1}[0,1]$ and, particularly, $f(0)=0$, we get
		\begin{equation}\label{ec8}
			\begin{split}
				|f(\varphi_1(y))|&=\frac{|f(\varphi_1(y))|}{|\varphi_1(y)|}\cdot|\varphi_1(y)|\leq \sup \left\{\frac{|f(p)-f(q)|}{|p-q|}: \ p,q \in [0,1], \ p\neq q\right\}|\varphi_1(y)|\\
				&\leq \|f\|\|\varphi_1\|.
			\end{split}
		\end{equation}
		Similarly to the inequality (\ref{ec8}), we can infer that 
		\begin{equation}\label{ec9}
			|f(\varphi_2(y))|\leq \|f\|\|\varphi_2\|.
		\end{equation}
		Putting the inequalities (\ref{ec5})-(\ref{ec9}) in (\ref{ec10}), it follows that
		\begin{equation}
			\begin{split}
				\frac{|(Tf)(x)-(Tf)(y)|}{|x-y|}&\leq \|\varphi\|\|f\|(\|\varphi_1\|- \varphi_1(0))\\
				&+\|\varphi\|\|f\|\|\varphi_2\|+\|\varphi\|\|f\|\|\varphi_1\|+\|\varphi\|\|f\|\|\varphi_2\|\\
				&= \|f\|\left[2\|\varphi\|(\|\varphi_1\|+\|\varphi_2\|)-\|\varphi\|\varphi_1(0)\right]<\infty
			\end{split}
		\end{equation}
		This proves that $Tf\in H_1[0,1]$ and, therefore, $Tf\in D^{0,1}[0,1]$.
		
		\item[b)] Taking into account a) and the fact that $(Tf)(0)=0$, it follows that
		\begin{equation}
			\|Tf\|\leq \|f\|\left[2\|\varphi\|(\|\varphi_1\|+\|\varphi_2\|)-\|\varphi\|\varphi_1(0)\right],
		\end{equation}
		which proves b).
		
		\item[c)] Take $x,y\in [0,1]$ with $x\neq y$. Since $d(Tf,Tg)=\|Tf-Tg\|=\|T(f-g)\|$, we can write
		\begin{equation}
			\begin{split}
				&\frac{|T(f-g)(x)-T(f-g)(y)|}{|x-y|}=\frac{1}{|x-y|}\left|\varphi(x)(f-g)(\varphi_1(x))+(1-\varphi(x))(f-g)(\varphi_2(x))\right.\\
				&\left.-\varphi(y)(f-g)(\varphi_1(y))-(1-\varphi(y))(f-g)(\varphi_2(y))\right|\\
				&\leq \frac{|\varphi(x)-\varphi(y)|}{|x-y|}|(f-g)(\varphi_1(x))|+|\varphi(y)|\frac{|(f-g)(\varphi_1(x))-(f-g)(\varphi_1(y))|}{|x-y|}\\
				&+\frac{|(1-\varphi(x))-(1-\varphi(y))|}{|x-y|}|(f-g)(\varphi_2(x))|+|1-\varphi(y)|\frac{|(f-g)(\varphi_2(x))-(f-g)(\varphi_2(y))|}{|x-y|}.
			\end{split}
		\end{equation}
		Now, by taking into account that $(f-g)(\varphi_1(1))=(f-g)(\varphi_2(0))=0$, we deduce the following chain of inequalities
		\begin{equation}
			\begin{split}
				\frac{|T(f-g)(x)-T(f-g)(y)|}{|x-y|}
				&\leq \|\varphi\|\frac{|(f-g)(\varphi_1(x))-(f-g)(\varphi_1(1))|}{|\varphi_1(x)-\varphi_1(1)|}\frac{|\varphi_1(x)-\varphi_1(1)|}{1-x}(1-x)\\
				&+ \|\varphi\|\frac{|(f-g)(\varphi_1(x))-(f-g)(\varphi_1(y))|}{|\varphi_1(x)-\varphi_1(y)|}\frac{|\varphi_1(x)-\varphi_1(y)|}{|x-y|}\\
				&+	\|\varphi\|\frac{|(f-g)(\varphi_2(x))-(f-g)(\varphi_2(0))|}{|\varphi_2(x)-\varphi_2(0)|}\frac{|\varphi_2(x)-\varphi_2(0)|}{x}\cdot x\\
				&	+ \|\varphi\|\frac{|(f-g)(\varphi_2(x))-(f-g)(\varphi_2(y))|}{|\varphi_2(x)-\varphi_2(y)|}\frac{|\varphi_2(x)-\varphi_2(y)|}{|x-y|}\\
				&\leq 2\|\varphi\|\|f-g\|(\|\varphi_1\|-\varphi_1(0))+2\|\varphi\|\|f-g\|\|\varphi_2\|\\
				&\leq \left[ 2\|\varphi\|\left(\|\varphi_1\|-\varphi_1(0))+\|\varphi_2\|\right)\right]\|f-g\|.
			\end{split}
		\end{equation}
		From this we infer that 
		\begin{equation}
			d(Tf,Tg)\leq \left[ 2\|\varphi\|\left(\|\varphi_1\|-\varphi_1(0))+\|\varphi_2\|\right)\right]d(f,g).
		\end{equation}
		This completes the proof of c) and concludes the reasoning. 
	\end{itemize}
\end{proof}
As a consequence of Theorem 1 we deduce the following result.
\begin{theorem}\label{teo2}
	Under the assumptions of Theorem \ref{teo1}, if 
	\begin{equation}
		\|\varphi\|\left(\|\varphi_1\|-\varphi_1(0))+\|\varphi_2\|\right)<\frac{1}{2},
	\end{equation} 
	the operator $T$ defined in Theorem \ref{teo1} has a unique fixed point $f^{\star}$ in $D^{0,1}[0,1]$. Furthermore, if we take an arbitrary $f_0\in D^{0,1}[0,1]$ then the iteration $(f^n)$ in $D^{0,1}[0,1]$ defined by 
	\begin{equation}
		(f^n)(x)=\varphi(x)f^{n-1}(\varphi_1(x))+(1-\varphi(x))f^{n-1}(\varphi_2(x)),
	\end{equation}
	for any $n\in \mathbb{N}$, converges in the metric $d$ to the unique fixed point $f^{\star}$ of $T$.
\end{theorem}
\begin{proof}
	Taking into account Theorem \ref{teo1} and the Banach contraction principle, we get the desired result.
\end{proof}
\begin{remark}
	A very important fact is that our main equation (\ref{ec1}), that is, 
	\begin{equation}
		f(x)=\varphi(x)f(\varphi_1(x))+(1-\varphi(x))f(\varphi_2(x)),
	\end{equation}
	has as solution $f(x)\equiv 0$, that is the trivial one. Note that the solution $f^{\star}$, given by Theorem \ref{teo2}, cannot be this function since $f^{\star}\in D^{0,1}[0,1]$ which implies that $f^{\star}(1)=1$.
\end{remark}

In the paper \cite{2n}, authors studied the existence and uniqueness of solutions of the following functional equation 
\begin{equation}\label{ec11}
	f(x)=xf(\alpha x+1-\alpha)+(1-x)f(\beta x)\quad 0<\alpha\leq \beta<1.
\end{equation}
They proved the following result.
\begin{theorem}[\cite{2n}, Theorem 5.2]\label{teo3}
	Suppose that $0<\alpha\leq \beta<1$ with $4\beta <1$. Then (\ref{ec11}) has a unique solution in $D^{0,1}[0,1]$.
\end{theorem}
\begin{remark}
	Notice that (\ref{ec11}) is a particular case of (\ref{ec1}), where $\varphi(x)=x$, $\varphi_1(x)=\alpha x+1-\alpha$, and $\varphi_2(x)=\beta x$. Furthermore, in this case, it is clear that $\varphi \in D^{0,1}[0,1]$ with $\|\varphi\|=1$, $\varphi_1, \varphi_2:[0,1]\to [0,1]$ with $\varphi_1(1)=1$, $\varphi_2(0)=0$ and $\|\varphi_1\|=1$, $\|\varphi_2\|=\beta$, $\varphi_1(0)=1-\alpha$. Therefore, 
	\begin{equation}
		\|\varphi\|(\|\varphi_1\|-\varphi_1(0)+\|\varphi_2\|)=1-(1-\alpha)+\beta=\alpha+\beta<2\beta<\frac{1}{2}.
	\end{equation}
	Consequently, under the assumptions used in \cite{2n}, equation (\ref{ec11}) can be treated by Theorem \ref{teo2}. On the other hand, the authors of \cite{2n} presented a corrigendum of this paper in \cite{7n}, where they studied the existence and uniqueness of solutions to (\ref{ec11}) under the same conditions of Theorem \ref{teo3} of \cite{2n} but in space $D^0[0,1]$ (in our notation), defined by
	\begin{equation}
		D^0[0,1]=\left\{f:[0,1]\to \mathbb{R}:\ \sup_{x\neq y}\frac{|f(x)-f(y)|}{|x-y|}<\infty \ \text{and}\ f(0)=0\right\}.
	\end{equation} 
	This corrigendum was motivated by the fact that $D^{0,1}[0,1]$ is not a Banach space. However, we must stress that since the authors of \cite{7n} worked in $D^{0}[0,1]$, the obtained unique solution to (\ref{ec11}) is actually the \emph{trivial solution}. On the other hand, in the present paper, since our solution lives in $D^{0,1}[0,1]$, the unique solution $f^\star$ cannot be the trivial one, since $f^{\star}(1)=1$.
\end{remark}

\section{Practical remarks and numerical examples}
In this section, we will focus on some practical aspects of solving functional equations that we have analyzed. Since in the original setting, the unique solution represents a probability distribution function, it is important to be able to find it relatively easily and cheaply. We will see that this task is usually far from trivial. 

\subsection{Paradise fish equation}
Let us once again go back to the equation that motivated our considerations, that is, \eqref{ec1n}. We will investigate some approximate solutions to this equation. The authors of \cite{2n} proposed that the solution $f$ can be found by the following iteration
\begin{equation}
	f^{n+1}(x) = x f^n(\alpha x + 1 - \alpha) + (1-x) f^n(\beta x),
\end{equation}
with some suitably chosen initial condition. An exemplary graph of iterations is shown in Fig. \ref{fig:Plots} where we have chosen two sets of parameters: $(\alpha,\beta) = (0.1,0.5)$ and $(\alpha,\beta) = (0.3,0.7)$ with an initial condition $f^0(x) = x$. We can see that the larger the $\beta$ the graph becomes more flat near $x = 1$. This is consistent with the exact limiting solution for $\beta = 1$, that is, $f(x) = 1$ for $x\neq 0$. Moreover, in Fig. \ref{fig:Error} on the left, we can also see the mean-square error between subsequent iterations, that is, $\|f^{n+1}-f^n\|$, plotted on a semi-logarithmic scale this time with one parameter set $(\alpha,\beta) = (0.1,0.5)$ and the initial condition $f^0(x) = \sin(\pi x/2)$. We can immediately see that, as expected by the fixed-point iteration, it decreases exponentially with ratio $e^{-0.62} \approx 0.54$, that is, the error decreases twice in each iteration, which confirms the classical results concerning Picard's iteration.   

\begin{figure}
	\centering
	\includegraphics[scale = 0.9]{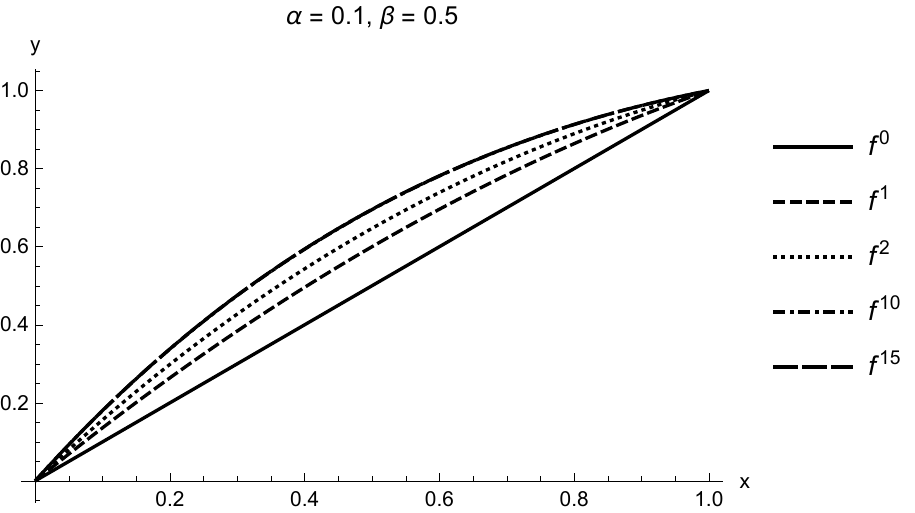}\vspace{24pt}
	\includegraphics[scale = 0.9]{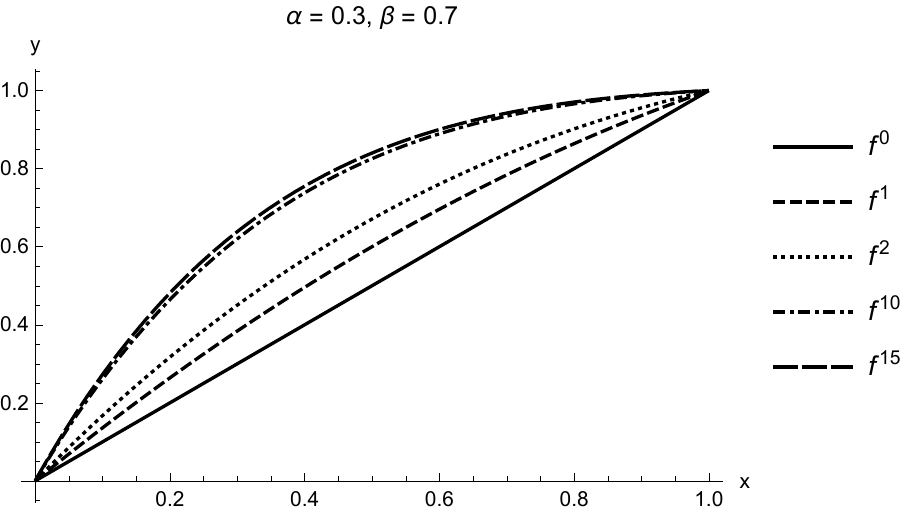}
	\caption{Exemplary plots of iterations converging to the solution of \eqref{ec1n}. Here, $(\alpha,\beta) = (0.1,0.5)$ (top), and $(\alpha,\beta) = (0.3,0.7)$ (bottom) with $f^0(x) = x$. }
	\label{fig:Plots}
\end{figure}

\begin{figure}
	\centering
	\includegraphics[scale = 1]{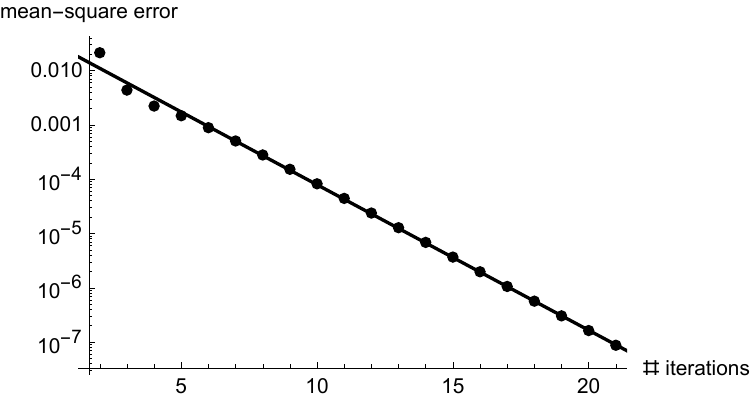}\vspace{24pt}
	\includegraphics[scale = 1]{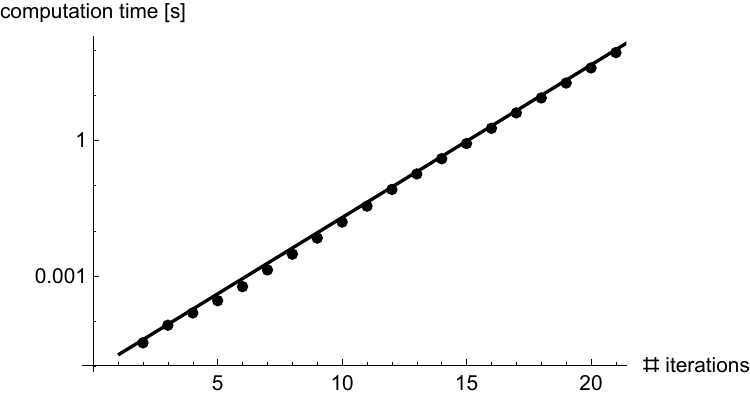}
	\caption{Top: least-squares error between subsequent iterations $n$ - a given number of iterations (dots) and the best-fit function $3.75 \times 10^{-2} \times e^{-0.62 n}$ (solid line). Bottom: time of computation required to obtain $n$ - a given number of iterations (dots) and the best-fit function $8.33 \times 10^{-6} \times 2^{1.12 n}$ (solid line). Here, $\alpha = 0.1$, $\beta = 0.5$, and $f^0(x) = \sin(\pi x /2)$. }
	\label{fig:Error}
\end{figure}

In practice, however, one would like to implement this iteration to obtain an approximate solution after several steps. As our computations in Wolfram Mathematica show, these iterations are very expensive and hence impractical for everyday use. Moreover, the symbolic output of iterates is very complex and not revealing. For example, in Fig. \ref{fig:Error} on the right, we have plotted the calculation times needed to obtain a certain number of iterations. Note the semi-logarithmic scale indicating exponential temporal complexity of the calculation. The best fit is given by $8.33 \times 10^{-6} \times 2^{1.12 n}$ which confirms the doubling of floating point operations in each iteration (by \eqref{ec1n} the right-hand side requires evaluating $f$ twice for different arguments). Due to this high cost of computation, it is, therefore, sensible to look for some simple approximations of the solution to \eqref{ec1n}. The following result gives one such function. We would like to note that it is relatively easy but tedious to obtain an optimal least-squares approximation that minimizes the residue of equation \eqref{ec1n}. However, the coefficients of the approximation are very complicated in this case, contradicting our main aim of obtaining a function of a simple form. Therefore, we present a suboptimal approximation with a neat form and decent error that can be controlled in a known way with $\alpha$ and $\beta$. 

\begin{theorem}
	For $x\in[0,1]$ define
	\begin{equation}\label{eqn:Approx}
		\widetilde{f}(x) := \frac{x(x+b)}{1+b}, \quad b = - \frac{(2-\alpha)^2+\beta^2-2}{2(\beta-\alpha)},
	\end{equation}
	and assume that $0<\alpha<\beta$ and $0<\beta\leq 2-\sqrt{2-\alpha^2}$. Then $\widetilde{f}$ is an increasing concave quadratic that satisfies $\widetilde{f}(0) = 0$ and $\widetilde{f}(1) = 1$. Moreover, its residue in the least-squares sense satisfies
	\begin{equation}\label{eqn:ApproxRes}
		\begin{split}
			&\min_{\substack{0<\alpha\leq\beta<1\\ \widetilde{f} \text{- quadratic}}}\left(\int_0^1 \left(\widetilde{f}(x) - (x \widetilde{f}(\alpha x + 1 - \alpha) + (1-x) \widetilde{f}(\beta x))\right)^2 dx \right)^\frac{1}{2} \\
			&\leq \frac{\beta^2-\alpha^2}{2\sqrt{210}} \leq \frac{(2-\sqrt{2-\alpha^2})^2-\alpha^2}{2\sqrt{210}} \leq \frac{(2-\sqrt{2})^2}{2\sqrt{210}} \approx 0.01.
		\end{split}
	\end{equation}
\end{theorem} 
\begin{proof}
	Note that by definition $\widetilde{f}$ satisfies the boundary conditions. Moreover, as a quadratic, it is convex for $b < -1$ and has a maximum greater than or equal to $1$ for $b \leq -2$. Therefore, the only admissible values of the parameter $b$ lie in the interval $(-\infty, -2]$. Furthermore, by a straightforward calculation the residue for the ansatz $\widetilde{f}(x) = x(x+b)/(1+b)$ is
	\begin{equation}
		\begin{split}
			|\widetilde{f}(x) - (x \widetilde{f}(\alpha x &+ 1 - \alpha) + (1-x) \widetilde{f}(\beta x))| \\
			&= \frac{x(1-x)}{-b-1} |(\beta^2-\alpha^2)x +1-(2+b)\alpha + \alpha^2 + b \beta| \\
			&\leq x(1-x) |(\beta^2-\alpha^2)x +1-(2+b)\alpha + \alpha^2 + b \beta| =: R(x,b)
		\end{split}
	\end{equation}
	since $b \leq -2$. Hence, by integration we obtain
	\begin{equation}
		\begin{split}
			\int_0^1 &\left(\widetilde{f}(x) - (x \widetilde{f}(\alpha x + 1 - \alpha) + (1-x) \widetilde{f}(\beta x))\right)^2 dx \\
			&\leq \int_0^1 x^2(1-x)^2 \left((\beta^2-\alpha^2)x +1-(2+b)\alpha + \alpha^2 + b \beta\right)^2 dx := P(b).
		\end{split}
	\end{equation}
	Now, we can compute the derivative of the estimate for the residue
	\begin{equation}
		\frac{d P}{db}(b) = \frac{1}{30}\frac{(\beta-\alpha)(2-2(2+b)\alpha + \alpha^2 + 2 b \beta + \beta^2)}{1+b}.
	\end{equation}
	Setting $dP/db = 0$, it is easy to see that the unique critical point $b=b_c$ is given in \eqref{eqn:Approx}. Moreover, the second derivative is
	\begin{equation}
		\frac{d^2 P}{db^2}(b_c) = \frac{(\beta-\alpha)^2}{15} > 0
	\end{equation}
	which is manifestly positive and we can conclude that $b_c$ is the unique local minimum. We also have to check whether this extremum is consistent with our requirement that $b \leq -2$ (monotonicity of the approximation). To this end, observe that since $\alpha < \beta$
	\begin{equation}
		b_c = - \frac{(2-\alpha)^2+\beta^2-2}{2(\beta-\alpha)} \leq -2 \iff \alpha^2 + (\beta-2)^2 \geq 2,
	\end{equation}
	that is the set admissible parameters $(\alpha,\beta)$ is the intersection of exterior of the circle with origin at $(0,2)$ and radius $\sqrt{2}$ with the cone $0<\alpha\leq \beta$. In particular, $\beta \leq 2-\sqrt{2-\alpha^2}$. 
	
	Finally, we can compute the estimate on the residue for the local minimum which by a simple calculation yields
	\begin{equation}
		\int_0^1 R(x, b_c)^2 dx = \frac{(\beta-\alpha)^2 (\beta+\alpha)^2}{840} \leq \frac{(\beta^2-\alpha^2)^2}{840} \leq \frac{\left((2-\sqrt{2-\alpha^2})^2-\alpha^2\right)^2}{840},
	\end{equation}
	since $\beta \leq 2-\sqrt{2-\alpha^2}$ in the worst case. This ends the proof. 
\end{proof}

As we have seen from the proof, our approximation $\widetilde{f}$ is only suboptimal, however, it has a very simple form with an explicit and revealing formula for the residue. In \emph{the worst case} the error is of the order of $10^{-2}$ which in practical applications may be sufficient. Moreover, when $\alpha \rightarrow \beta^-$ the error vanishes and the approximation $\widetilde{f}(x) \rightarrow x$ uniformly. As can be checked directly, the identity function is an exact solution to the main equation \eqref{ec1n} in this case. Therefore, the approximation is optimal in that sense. To get a quantitative image of the approximation's performance, we can observe Fig. \ref{fig:Approximation}, where we have plotted two approximations: $\widetilde{f}$ from \eqref{eqn:Approx} and the least-squares (due to the complexity of the formula we do not reproduce it here) along with $f^{15}$ as a proxy for the exact solution. First of all, both approximations are nearly identical, which has also been confirmed by other numerical computations. We can see that even for the limiting value of $\beta = 2-\sqrt{2-\alpha^2}\approx 0.617$ we obtain decent precision with an error of $0.018$ and residue $0.01$ that is consistent with our estimate. On the other hand, for $\beta = 0.5$ the error is one order of magnitude smaller, that is $4.7\times 10^{-3}$ with a residue $3\times 10^{-3}$. We can conclude that $\widetilde{f}$ provides a decent and simple almost least-squares optimal approximation to the fixed-point of \eqref{ec1n}. Of course, the search for better approximations can be carried over to some other classes of functions for which the region of parameter admissibility can be extended to the whole space $0<\alpha\leq\beta<1$. We have obtained some initial good results for the exponential family, however, the details will be the subject of future work. 

\begin{figure}
	\centering
	\includegraphics[scale = 0.9]{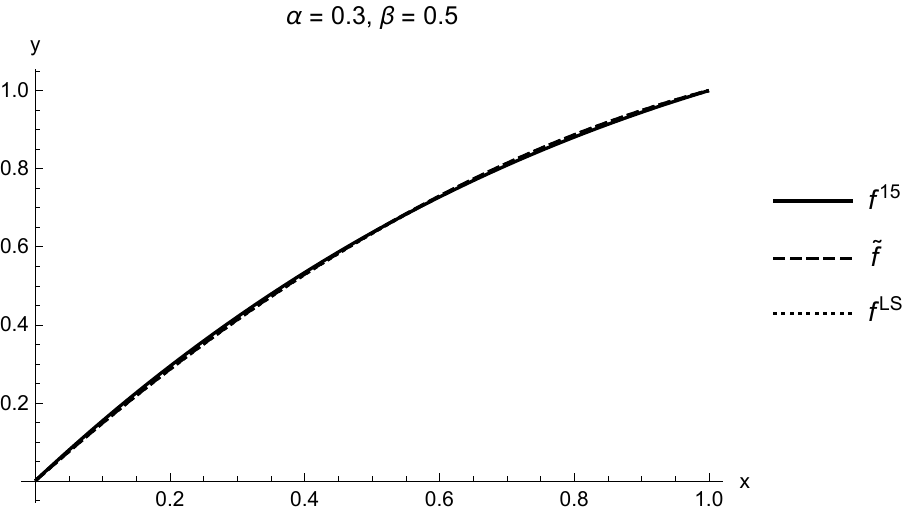}\vspace{24pt}
	\includegraphics[scale = 0.9]{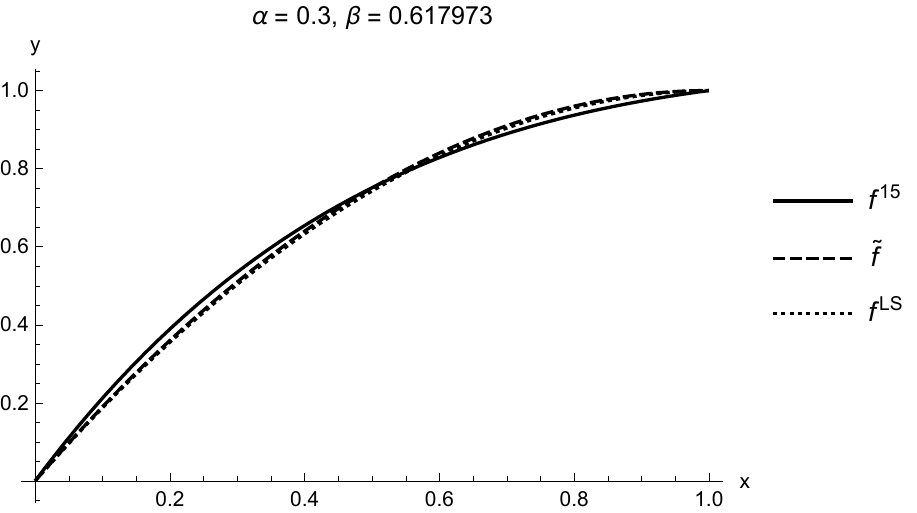}
	\caption{A proxy of the exact solution $f^{15}$ along the two approximations: $\widetilde{f}$ and $f^{LS}$ for two different sets of parameters and the initial condition $f^0(x) = x$. }
	\label{fig:Approximation}
\end{figure}

\subsection{A general equation with exact solution}
Considering a general family of admissible general equations, we have the possibility of computing the error of successive approximations. Previously, we only estimated it by computing the difference between neighboring iterations. Consider
\begin{equation}
	\varphi_m(x) = \frac{(1-\beta^m)x^m}{(\alpha x +(1-\alpha))^m - \beta^m x^m} \in D^{0,1}[0,1].
\end{equation}
It is then straightforward to check that the function 
\begin{equation}\label{eqn:ExactSol}
	f(x) = x^m
\end{equation}
with $m>0$ is a solution of the following equation in $D^{0,1}[0,1]$ 
\begin{equation}\label{eqn:ExactSolEq}
	f(x) = \varphi_m(x) f(\alpha x + 1-\alpha) + (1-\varphi_m(x)) f(\beta x).
\end{equation}
Since the solution is known and independent on $\alpha$ and $\beta$ we can easily compute the error of successive approximations. In Fig. \ref{fig:ExactSol} we can see some exemplary iterations along with their true error $\|f-f^n\|$ (previously, when the solution was not known explicitly, we were only able to find the error between successive iterations). The error behaves exponentially decreasing by a factor of $1.4$ in each iteration. It is relatively straightforward to implement some other functional equations belonging to the considered family. However, the examples presented convey the main features of their numerical solution. 

\begin{figure}
	\centering
	\includegraphics[scale = 0.9]{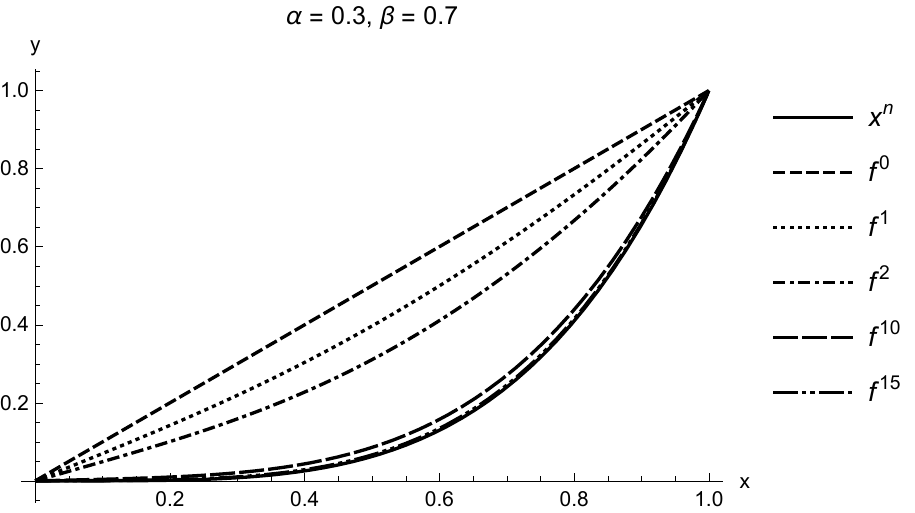}\vspace{24pt}
	\includegraphics[scale = 1.1]{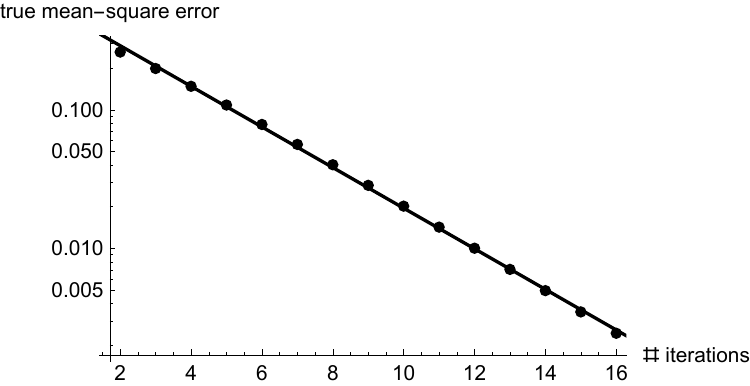}
	\caption{Exemplary plots of iterations converging to the solution of \eqref{ec1n} (top) and the corresponding error $\|f-f^n\|$ (bottom). Here, $(\alpha,\beta) = (0.3,0.7)$ and $m=4$ with $f^0(x) = x$. The fitted error curve is $0.57\times e^{-0.34 n}$.}
	\label{fig:ExactSol}
\end{figure}

\section*{Acknowledgements}
J. Caballero and K. Sadarangani are partially supported by the project PID2019-106093GB-I00. \L.P. has been supported by the National Science Centre, Poland (NCN) under the grant Sonata Bis with a number NCN 2020/38/E/ST1/00153.

\bibliography{biblio}

\begin{thebibliography}{1}

\bibitem{1n}
J{\'o}zef Bana{\'s} and Rafa{\l} Nalepa.
\newblock On the space of functions with growths tempered by a modulus of
  continuity and its applications.
\newblock {\em Journal of Function Spaces and Applications}, 2013, 2013.

\bibitem{berinde2015functional}
Vasile Berinde and Abdul~Rahim Khan.
\newblock On a functional equation arising in mathematical biology and theory
  of learning.
\newblock {\em Creative Mathematics and Informatics}, 24(1):9--16, 2015.

\bibitem{bush1956two}
Robert~R Bush and Thurlow~R Wilson.
\newblock Two-choice behavior of paradise fish.
\newblock {\em Journal of Experimental Psychology}, 51(5):315, 1956.

\bibitem{istruactescu1976functional}
Vasile~I Istr{\u{a}}{\c{t}}escu.
\newblock On a functional equation.
\newblock {\em Journal of Mathematical Analysis and Applications},
  56(1):133--136, 1976.

\bibitem{lyubich1973functional}
I~Yu Lyubich and AP~Shapiro.
\newblock On a functional equation.
\newblock {\em Teor. Funkts., Funkts. Anal. Prilozh}, 17:81--84, 1973.

\bibitem{2n}
Ali Turab and Wutiphol Sintunavarat.
\newblock On analytic model for two-choice behavior of the paradise fish based
  on the fixed point method.
\newblock {\em Journal of Fixed Point Theory and Applications}, 21:1--13, 2019.

\bibitem{7n}
Ali Turab and Wutiphol Sintunavarat.
\newblock Corrigendum: On analytic model for two-choice behavior of the
  paradise fish based on the fixed point method.
\newblock {\em Journal of Fixed Point Theory and Applications}, 22(82):1--3,
  2020.

\end{thebibliography}
\bibliographystyle{plain}

\end{document}